\documentclass[11pt,english]{amsart}
\usepackage[T1]{fontenc}
\usepackage[latin9]{inputenc}
\usepackage{babel}
\usepackage{amsthm}
\usepackage{amstext}
\usepackage{amssymb}
\usepackage{setspace}
\usepackage[unicode=true,
 bookmarks=true,bookmarksnumbered=false,bookmarksopen=false,
 breaklinks=false,pdfborder={0 0 1},backref=false,colorlinks=false]
 {hyperref}
\hypersetup{pdftitle={Essential normality and the decomposability of algebraic varieties},
 pdfauthor={Matthew Kennedy and  Orr Shalit}}

\makeatletter

\numberwithin{equation}{section}
\numberwithin{figure}{section}
\theoremstyle{plain}
\newtheorem{thm}{\protect\theoremname}[section]
  \theoremstyle{plain}
  \newtheorem{conjecture}[thm]{\protect\conjecturename}
  \theoremstyle{plain}
  \newtheorem{prop}[thm]{\protect\propositionname}
  \theoremstyle{remark}
  \newtheorem*{rem*}{\protect\remarkname}
  \theoremstyle{plain}
  \newtheorem{lem}[thm]{\protect\lemmaname}
  \theoremstyle{remark}
  \newtheorem{rem}[thm]{\protect\remarkname}

\usepackage[T1]{fontenc}
\usepackage{ae,aecompl}

\usepackage{microtype}

\makeatother

  \providecommand{\conjecturename}{Conjecture}
  \providecommand{\lemmaname}{Lemma}
  \providecommand{\propositionname}{Proposition}
  \providecommand{\remarkname}{Remark}
\providecommand{\theoremname}{Theorem}

\begin{document}

\title[Essential normality and decomposability]{Essential normality and the decomposability of algebraic varieties}

\author{Matthew Kennedy}

\address{School of Mathematics and Statistics, Carleton University, 1125 Colonel
By Drive, Ottawa, Ontario K1S 5B6, Canada}

\email{mkennedy@math.carleton.ca}

\author{Orr Moshe Shalit}

\address{Department of Mathematics, Ben-Gurion University of the Negev, P.O.B.
653, Beer-Sheva, 84105, Israel}

\email{oshalit@math.bgu.ac.il}
\begin{abstract}
We consider the Arveson-Douglas conjecture on the essential normality of homogeneous submodules corresponding to algebraic subvarieties of the unit ball. We prove that the property of essential normality is preserved by
isomorphisms between varieties, and we establish a similar result for maps between varieties that are not necessarily invertible. We also relate the decomposability of an algebraic variety to the problem of establishing the essential normality of the corresponding submodule. These results are applied to prove that the Arveson-Douglas conjecture holds for submodules corresponding to varieties that decompose into linear subspaces, and varieties that decompose into components with mutually disjoint linear spans.
\end{abstract}


\thanks{First author partially supported by NSERC Canada}
\thanks{Second author partially supported by Israel Science Foundation Grant no. 474/12 and by the European Union Seventh Framework Programme ({\em FP7/2007-2013}) under grant agreement no. 321749}

\maketitle

\section{Introduction}
In this paper, we consider a conjecture of Douglas and Arveson that implies a correspondence between algebraic varieties and C*-algebras of essentially normal operators. In the papers \cite{Sha11} and \cite{Ken12}, we showed that this conjecture can be viewed as a problem of finding certain nice decompositions of submodules of $\mathbb{C}[z_1,\ldots,z_d]$. In the present paper, we take a slightly different perspective, and relate the conjecture to the geometry of the variety in question.

Let $d$ be a fixed positive integer, and let $\mathbb{C}[z]=\mathbb{C}[z_{1},\ldots,z_{d}]$
denote the algebra of complex polynomials in $d$ variables. The Drury-Arveson space $H_{d}^{2}$ is
the reproducing kernel Hilbert space on the unit ball $\mathbb{B}_{d}$
generated by the kernel functions 
\[
k_{\lambda}(z)=\frac{1}{1-\langle z,\lambda\rangle},\quad\lambda\in\mathbb{B}_{d}.
\]
Equivalently, $H_d^2$ is the completion of $\mathbb{C}[z]$ with respect to the inner product
\[
\langle z^\alpha,z^\beta\rangle = \delta_{\alpha,\beta}\frac{\alpha_1! \cdots \alpha_d!}{(\alpha_1+\ldots +\alpha_d)!}, \quad \alpha, \beta\ \in \mathbb{N}_0^d,
\]
where we have used the notation $z^\alpha = z_1^{\alpha_1} \cdots z_d^{\alpha_d}$ for $\alpha=(\alpha_1,\ldots,\alpha_d)$ in $\mathbb{N}_0^d$.

The $d$-shift $S=(S_{1},\ldots,S_{d})$  is the $d$-tuple of multiplication operators on $H_d^2$ corresponding to the coordinate functions $z_1,\ldots,z_d$. They act by
\[
(S_{i}f)(z)=z_{i}f(z),\quad f\in H_{d}^{2}.
\]
We will be particularly interested in these operators, which were introduced and extensively studied in \cite{Arv98}. Together with the $d$-shift $S$, the space $H_d^2$ forms a Hilbert module over
$\mathbb{C}[z]$, with the module action given by
\[
pf = p(S_1,\ldots,S_d)f, \quad p \in \mathbb{C}[z],\ f \in H_d^2.
\]
Endowed with this module structure, $H_d^2$ is known as the {\em $d$-shift Hilbert module}.

For an ideal $I$ of $\mathbb{C}[z]$, we define
\[
\mathcal{F}_{I}=H_{d}^{2}\ominus I.
\]
Note that since the closure of $I$ in $H_d^2$ is an invariant subspace for each $S_j$, the space $\mathcal{F}_I$ is coinvariant for each $S_j$.
We let $S_{j}^{I}$ denote the compression of $S_{j}$ to $\mathcal{F}_{I}$, i.e.,
\[
S_{j}^{I}=P_{\mathcal{F}_{I}}S_{j}\mid_{\mathcal{F}_{I}}.
\]
Then as a Hilbert module, $\mathcal{F}_{I}$ is equivalent to the quotient of $H_{d}^{2}$ by the closure of $I$ in $H_{d}^{2}$.

We will require the following correspondence between ideals of $\mathbb{C}[z]$ and subsets of the unit ball $\mathbb{B}_d$  of $\mathbb{C}^d$. For an ideal $I$ of $\mathbb{C}[z]$, we define
\[
V(I)=\{z\in\mathbb{B}_{d}\mid p(z)=0\ \forall p\in I\},
\]
and for a subset $V$ of $\mathbb{B}_{d}$, we define
\[
I(V)=\{p\in\mathbb{C}[z]\mid p(z)=0\ \forall z\in V\}.
\]
For a homogeneous ideal $I$ we shall call the set $V(I)$ a {\em homogeneous variety in $\mathbb{B}_d$}. All the varieties in this paper will be homogeneous varieties in $\mathbb{B}_d$. 

If the ideal $I$ is radical, then the space $\mathcal{F}_{I}$ is a reproducing kernel Hilbert space over $V(I)$. More generally, it was established in \cite[Lemma 5.5]{DRS11} that in this case we have the equality
\[
\mathcal{F}_{I}=\overline{\mathrm{span}}\{k_{\lambda}\mid\lambda\in V(I)\}.
\]

Arveson's conjecture is that for every homogeneous ideal $I$ of $\mathbb{C}[z]$, the quotient operators $S_1^I,\ldots,S_d^I$ satisfy
\begin{equation}
[S_{i}^{I},S_{j}^{I*}]:=S_{i}^{I} S_{j}^{I*} - S_{j}^{I*} S_{i}^{I} \in\mathcal{L}^{p},\quad p>d,\ 1 \leq i,j \leq d, \label{eq:trace}
\end{equation}
where for $1\leq p \leq \infty$, $\mathcal{L}^p$ denotes the set of Schatten $p$-class operators on $H_d^2$.
The general version of Arveson's conjecture includes multiplicity, but we are not worrying about that for now, and in fact, by \cite[Section 5]{Sha11}, the full conjecture is equivalent to the scalar case (up to a small modification of the range of $p$).

Douglas conjectured further that (\ref{eq:trace}) should hold for all $p>\dim I$. Note that $\dim I$ is defined in the following way. It
is known that there is a polynomial $h_{I}(x)$, called the {\em Hilbert polynomial}, such that for sufficiently large $n$, the dimension of $\mathbb{H}_{n}\ominus I_{n}$ is equal
to $h_{I}(n)$. The dimension $\dim I$ is defined to be $\deg(h_{I}(x))+1$ (see, e.g., \cite[Chapter 9]{CLS92}). If
$V$ is the affine variety determined by $I$ then $\dim I = \dim V$. For example, when the variety $V$ is a union of subspaces this is
just the maximal dimension of the subspaces.

In this note we will be concerned with the Arveson-Douglas conjecture
for radical homogeneous ideals. To express our ideas in the clearest way, we are
led to introduce the following notation. If $X$ is a subspace of $\mathbb{C}^{d}$,
then we write $X^{n}$ for the $n$-th symmetric tensor power of $X$ with itself. If $V\subseteq X$ is a homogeneous variety in the ball, i.e. if $V$ is of the form $V = V(I)$, for some radical homogeneous ideal $I$ of $\mathbb{C}[z]$, then we define $V^{n}$ to be the subspace of $X^{n}$ spanned by elements of the form
\[
\lambda^{n}=\underbrace{\lambda\otimes\cdots\otimes\lambda}_{n\textrm{ times }}, \quad \lambda \in V.
\]
Thus, if $V=V_{1}\cup\ldots\cup V_{k}$ is a union of varieties, then we have that 
\[
V^{n}=\sum_{i=1}^{k}V_{i}^{n}.
\]
Using the natural identification of $\mathbb{C}[z]$ with symmetric Fock space gives the decomposition
\[
\mathcal{F}_{I}=\oplus_{n=0}^{\infty}V^{n}.
\]
With this identification, the kernel functions $k_\lambda$ of $\mathcal{F}_I$ are of the form 
\[
k_{\lambda}=\sum_{n=0}^{\infty}\overline{\lambda}^{n}, \quad \lambda \in V.
\]
We remark that (for sufficiently large $n$) the dimension of $V^{n}$
is bounded by $n^{d-1}$, because it is a subspace of $(\mathbb{C}^{d})^{n}$,
which has dimension $\frac{(n+d-1)!}{n!(d-1)!}$.

When we consider $\mathcal{F}_{I}$ as a reproducing kernel Hilbert space over $V(I)$, then the operators $S_{i}^{I}$
correspond to multiplication operator $M_{f_i}$ defined by
\[
(M_{f_i} g)(z) = (f_{i}g)(z), \quad g \in \mathcal{F}_I,
\]
where $f_i = z_{i}\big|_{V(I)}$. The algebra $\mathcal{A}_{I}$ is defined to be the normed closed
unital algebra generated by $(S_{1}^{I},\ldots,S_{d}^{I})$. This algebra is a normed closed subalgebra of the multiplier algebra of
$\mathcal{F}_{I}$. If $p$ belongs to $\mathbb{C}[z]$, then it will be convenient to identify $p(S^I_1, \ldots, S^I_d)$ with the multiplication operator $M_p$.

For $p \geq 1$, we will say that the quotient module $\mathcal{F}_{I}$ is \emph{$p$-essentially normal }if 
\[
[S_{i}^{I},S_{j}^{I*}]\in\mathcal{L}^{p},\quad \ 1 \leq i,j \leq d.
\]
Recall that this is equivalent to $|[S_{i}^{I},S_{j}^{I*}]|^p$ being trace class for $1 \leq i,j \leq d$.

If $V=V(I)$ and $I=I(V)$, which is the case whenever $I$ is a radical ideal, then we will write $S_1^V, \ldots ,S_d^V$ for $S_1^I, \ldots, S_d^I$.
Similarly, we will write $\mathcal{F}_{V}$ for $\mathcal{F}_{I}$, and $\mathcal{A}_V$ for $\mathcal{A}_I$. Using this notation, we now state for reference the form of the Arveson-Douglas conjecture that we consider in this paper.

\begin{conjecture}[Geometric Arveson-Douglas Conjecture] \label{conj}
Let $V$ be a homogenous variety in $\mathbb{B}_{d}$. Then the submodule $\mathcal{F}_{V}$
is $p$-essentially normal for every $p>\dim V$.
\end{conjecture}

Note that the essential normality of $\mathcal{F}_{V}$ is independent of the ambient space $\mathbb{C}^{d}$ (and in particular of the
dimension $d$) in which we choose to (isometrically) embed $V$ (see \cite[Remark 8.1]{DRS11}). 

Conjecture \ref{conj} originated with Arveson's investigation of the curvature invariant of a commuting $d$-tuple \cite{Arv00, Arv02}. In the past decade, it has drawn a lot of attention, for example in the papers \cite{Arv05,Arv07,Dou06a,Dou06b,DS11,DW12,Esc11,GW08,Ken12,Sha11}, which deal directly with this conjecture. We also wish to mention two recent papers, \cite{DW11} and \cite{FX12}, which treat the essential normality of a principal ideal generated by a (not necessarily homogeneous) polynomial. These papers are worth mentioning, not only because the problem they treat is closely related, but also because they introduce promising analytic techniques that are quite different from previous approaches to the general problem of essential normality.

The main result of \cite[Section 7.3]{DRS11} is that if $V$ and $W$ are
``tractable" homogeneous varieties, and if
$A$ is an invertible linear map that maps $W$ onto $V$ that is
isometric on $W$, then the map $f\mapsto f\circ A$ is an isomorphism between the algebras
$\mathcal{A}_{V}$ and $\mathcal{A}_{W}$ \cite[Theorem 7.17]{DRS11}.
Furthermore, it was shown that this isomorphism is implemented by
a similarity $\tilde{A}^{*}$, i.e.
\[
\varphi(M_{f})=\tilde{A}^{*}M_{f}(\tilde{A}^{*})^{-1},
\]
where $\tilde{A}:\mathcal{F}_{W}\to\mathcal{F}_{V}$ is an invertible
bounded linear map satisfying
\[
\tilde{A}k_{\lambda}=k_{A\lambda}.
\]
Recently, in \cite{Har12}, Hartz was able to prove a stronger version of this result that does not require the varieties to be tractable. We will require this result for what follows.

In this paper, we study the Arveson-Douglas conjecture for submodules of the form $\mathcal{F}_V$, where $V$ is a homogeneous variety in $\mathbb{B}_d$.
In Section \ref{sec:maps}, we prove that if $W$ is a homogeneous variety in $\mathbb{B}_{d'}$, for some positive integer $d'$, and if $\mathcal{A}_V$ is isomorphic to $\mathcal{A}_W$, then $\mathcal{F}_V$ is $p$-essentially normal if and only if  $\mathcal{F}_W$ is $p$-essentially normal. We also establish a similar result for maps between varieties that are not necessarily isomorphic.

In Section \ref{sec:decomposability}, we consider when it is possible to decompose $V$ as $V=V_1\cup \ldots \cup V_n$, where $V_1,\ldots,V_n$ are homogeneous varieties in $\mathbb{C}^d$ with the property that the algebraic sum $\mathcal{F}_{V_1} + \ldots + \mathcal{F}_{V_n}$ is closed. This is a geometric analogue of the notion of the decomposability of a submodule that was introduced in \cite{Ken12}. We relate this geometric notion of decomposability to the problem of establishing the $p$-essential normality of the submodule $\mathcal{F}_V$.

Finally, in Section \ref{sec:applications}, we apply the results from Section 2 and Section 3 to establish the Arveson-Douglas conjecture for two new classes of examples. Using Hartz's result from \cite{Har12}, we prove that  $\mathcal{F}_V$ satisfies the Arveson-Douglas conjecture when $V$ decomposes as the union of linear subspaces. We also prove that $\mathcal{F}_V$ satisfies the Arveson-Douglas conjecture when $V$ decomposes into varieties $V_1,\ldots,V_n$ such that each $\mathcal{F}_{V_i}$ satisfies the conjecture, and $\mathrm{span}(V_i) \cap \mathrm{span}(V_j) = {0}$ whenever $i \ne j$. These are perhaps the simplest classes of examples for which the conjecture was not previously known to be true.

\section{Linear maps between varieties and essential normality} \label{sec:maps}

\subsection{Invertible maps}\label{subsec:invertible}
\begin{thm}\label{thm:isomorphism}
Let $V$ and $W$ be homogeneous varieties in $\mathbb{B}_{d}$ and
$\mathbb{B}_{d'}$ respectively. Suppose the algebras $\mathcal{A}_{V}$
and $\mathcal{A}_{W}$ are algebraically isomorphic. Then for $p\geq1$, $\mathcal{F}_{V}$
is $p$-essentially normal if and only if $\mathcal{F}_{W}$ is $p$-essentially
normal.\end{thm}
\begin{proof}
Since $\mathcal{A}_{V}$ and $\mathcal{A}_{W}$ are isomorphic, by
results of \cite{DRS11} (Proposition 7.1, Theorem 7.4 and Proposition 8.3) there is a linear transformation $A:\mathbb{C}^{d}\to\mathbb{C}^{d'}$ that maps $V$ bijectively onto $W$. It now follows from \cite{Har12} (Proposition 2.5 and Corollary 5.8) that there is an invertible linear map $\tilde{A}:\mathcal{F}_{V}\to\mathcal{F}_{W}$
defined by 
\[
\tilde{A}k_{\lambda}=k_{A\lambda} , \quad \lambda \in V.
\]
It follows that if $f$ is a polynomial in $\mathcal{F}_V$ then
\[
\tilde{A}f=f\circ A^{*}.
\]
The adjoint $\tilde{A}^{*}:\mathcal{F}_{W}\to\mathcal{F}_{V}$
is defined by
\[
\tilde{A}^{*}f=f\circ A,\quad f\in\mathcal{F}_{W}.
\]
Note that for a polynomial $f$ in $\mathcal{F}_{V}$, $\tilde{A}M_{f}=M_{f\circ A^{*}}\tilde{A}$,
and similarly for a  polynomial $f$ in $\mathcal{F}_{W}$, $\tilde{A}^{*}M_{f}=M_{f\circ A}\tilde{A}^{*}$. 

Fix polynomials $f$ and $g$ in $\mathcal{F}_{V}$. Then using the
identities $M_{g\circ A^{*}}=\tilde{A}M_{g}\tilde{A}^{-1}$, $M_{f\circ A^{*}}^{*}\tilde{A}=\tilde{A}M_{f\circ A^{*}A}^{*}$
 and $M_{f\circ A^{*}}^{*}=\tilde{A}M_{f\circ A^{*}A}^{*}\tilde{A}^{-1},$
we have 
\begin{eqnarray*}
M_{f\circ A^{*}}^{*}M_{g\circ A^{*}} & = & M_{f\circ A^{*}}^{*}\tilde{A}M_{g}\tilde{A}^{-1}\\
 & = & \tilde{A}M_{f\circ A^{*}A}^{*}M_{g}\tilde{A}^{-1}\\
 & = & \tilde{A}M_{g}M_{f\circ A^{*}A}^{*}\tilde{A}^{-1}+\tilde{A}[M_{f\circ A^{*}A},M_{g}]\tilde{A}^{-1}\\
 & = & \tilde{A}M_{g}\tilde{A}^{-1}\tilde{A}M_{f\circ A^{*}A}^{*}\tilde{A}^{-1}+\tilde{A}[M_{f\circ A^{*}A},M_{g}]\tilde{A}^{-1}\\
 & = & M_{g\circ A^{*}}M_{f\circ A^{*}}^{*}+\tilde{A}[M_{f\circ A^{*}A},M_{g}]\tilde{A}^{-1}.
\end{eqnarray*}
Therefore,
\[
[M_{f\circ A^{*}}^{*},M_{g\circ A^{*}}]=\tilde{A}[M_{f\circ A^{*}A},M_{g}]\tilde{A}^{-1}.
\]
and hence $[M_{f\circ A^{*}}^{*},M_{g\circ A^{*}}]$ belongs to $\mathcal{L}^{p}$
if and only if $[M_{f\circ A^{*}A}^{*},M_{g}]$ belongs to $\mathcal{L}^{p}$.
Letting $f$ and $g$ be suitable linear combinations of the coordinate functions one sees that $\mathcal{F}_W$ is $p$-essentially normal if and only if $\mathcal{F}_V$ is $p$-essentially normal. 
\end{proof}

\subsection{Maps that are not necessarily invertible}\label{subsec:not_invertible}
\begin{prop}\label{prop:quantitative}
Let $V$ and $W$ be homogeneous varieties in $\mathbb{B}_{d}$ and $\mathbb{B}_{d'}$, respectively, with decompositions into (not necessarily irreducible)
subvarieties $V=V_{1}\cup\ldots\cup V_{k}$ and $W=W_{1}\cup\ldots\cup W_{k}$ with the property that $\mathrm{span}(W_{i})\cap\mathrm{span}(W_{j})=\{0\}$
whenever $i\ne j$. Suppose that there is a linear map $A:\mathbb{C}^{d}\to\mathbb{C}^{d'}$
such that $A(V_{i})=W_{i}$ and such that the
restriction of $A$ to $\mathrm{span}(V_{i})$ is isometric for all $i$,  $1\leq i\leq k$. Then the map 
defined by\textup{
\[
\tilde{A}k_\lambda  =k_{A\lambda},\quad \lambda\in V
\]
} extends to a bounded linear map $\tilde{A}:\mathcal{F}_{V}\to\mathcal{F}_{W}$. Moreover, $\tilde{A}$ is the sum of a unitary operator and a trace class operator.
\end{prop}

\begin{rem*}
If $V_{i}$ is irreducible, and if $A$ is {\em isometric} on  $V_{i}$, then it follows from \cite[Proposition 7.6]{DRS11} that the restriction of $A$ to $\mathrm{span}(V_{i})$ is automatically isometric.\end{rem*}

\begin{proof}
It suffices to prove the lemma for the case when $V$ and $W$ are
unions of nontrivial subspaces (see the first paragraph of \cite[Theorem 7.16]{DRS11}). Hence we can suppose that $\{0\}\ne V_{i}=\mathrm{span}(V_{i})$
and $\{0\}\ne W_{i}=\mathrm{span}(W_{i})$.

The fact that the operator $\tilde{A}$ is bounded follows from the
results in \cite{Har12}. However, in order to prove that $\tilde{A}$
is the sum of a unitary operator and a trace class operator, we will
need to obtain quantitative estimates. If $M$ and $N$ are two subspaces of a Hilbert space then we denote (following \cite{Fri37})
\[
\cos(M,N) = \sup \{|\langle x, y \rangle | : x \in M \ominus (M\cap N), y \in N \ominus (M \cap N), \|x\|=\|y\|=1\}.
\]
By the finite-dimensionality of $V_{1},\ldots,V_{n}$, $\mathrm{cos}(V_{i},V_{j})<1$ and  $\mathrm{cos}(W_{i},W_{j})<1$ whenever $i\ne j$. Let 
\[
c=\max(\{\cos(V_{i},V_{j})\mid i\ne j\}\cup\{\cos(W_{i},W_{j})\mid i\ne j\}).
\]
Then $0\leq c<1$. For $v$ in $V^{n}$ and $w$ in $W^{n}$, write
$v=\sum_{i=1}^{k}v_{i}$ and $w=\sum_{i=1}^{k}w_{i}$, where each
$w_{i}$ belongs to $W_{i}^{n}$. Then as in the proof of \cite[Lemma 7.10]{DRS11}, for sufficiently large $n$, 
\begin{equation}
(1-kc^{n})\|v\|^{2}\leq\sum_{i=1}^{k}\|v_{i}\|^{2}\leq(1+kc^{n})\|v\|^{2}\label{eq:est-V}
\end{equation}
and
\begin{equation}
(1-kc^{n})\|w\|^{2}\leq\sum_{i=1}^{k}\|w_{i}\|^{2}\leq(1+kc^{n})\|w\|^{2}.\label{eq:est-W}
\end{equation}

The space $\mathcal{F}_{V}$ decomposes as $\mathcal{F}_{V}=\oplus_{n=0}^{\infty}V^{n}$,
and $\tilde{A}$ is defined on $V^{n}$ by setting 
\[
\tilde{A}\lambda^{n}=(A\lambda)^{n},\quad\lambda\in V,
\]
and extending by linearity. Since $W^{n}=W^{\otimes n}$, the operator
$\tilde{A}$ can also be realized as
\[
\tilde{A}=\oplus_{n=0}^{\infty}A^{\otimes n}.
\]
Therefore, by the hypothesis that $A$ is isometric on each $V_{i}$,
the restriction of $\tilde{A}$ to $V_{i}^{n}$ is a unitary from
$V_{i}^{n}$ to $W_{i}^{n}$.

As above, for $v$ in $V^{n}$ write $v=\sum_{i=1}^{k}v_{i}$, where
each $v_{i}$ belongs to $V_{i}^{n}$. Then by (\ref{eq:est-V}) and
(\ref{eq:est-W}), for sufficiently large $n$,
\begin{eqnarray}
\|\tilde{A}v\|^{2} & = & \|\sum_{i=1}^{k}A^{\otimes n}v_{i}\|^{2}\label{eq:est2}\\
 & \leq & \frac{1}{1-kc^{n}}\sum_{i=1}^{k}\|A^{\otimes n}v_{i}\|^{2}\nonumber \\
 & = & \frac{1}{1-kc^{n}}\sum_{i=1}^{k}\|v_{i}\|^{2}\nonumber \\
 & \leq & \frac{1+kc^{n}}{1-kc^{n}}\|v\|^{2},\nonumber 
\end{eqnarray}
By a similar argument, for sufficiently large $n$,
\begin{equation}
\|\tilde{A}v\|^{2}\geq\frac{1-kc^{n}}{1+kc^{n}}\|v\|^{2}.\label{eq:est3}
\end{equation}

Let $\tilde{A}=U|\tilde{A}|$ be the polar decomposition of $\tilde{A}$.
Since $\tilde{A}$ is graded, i.e. $\tilde{A}(V^{n})=W^{n}$, it follows
that $U$ and $|\tilde{A}|$ are also graded. Write $\tilde{A}=U+U(|\tilde{A}|-I)$.
Since $A$ (and hence $\tilde{A}$) is not necessarily invertible,
the partial isometry $U$ is not necessarily a unitary. However, by
(\ref{eq:est3}), the restriction of $\tilde{A}$ to $V^{n}$ is invertible
for sufficiently large $n$, so $U$ is a finite rank perturbation
of a unitary. Hence we will be done once we show that $|\tilde{A}|-I$
is a trace class operator.

The inequalities (\ref{eq:est2}) and (\ref{eq:est3}) are equivalent
to the existence of a constant $M>0$ such that for $v$ in $V^{n}$,
\[
(1-Mc^{n})\|v\|\leq\||\tilde{A}|v\|\leq(1+Mc^{n})\|v\|.
\]
Hence the eigenvalues of the restriction of $|\tilde{A}|$ to $V^{n}$
are contained in the interval $[1-Mc^{n},1+Mc^{n}]$, and it follows
that the eigenvalues of the restriction of $|\tilde{A}|-I$ to $V^{n}$
are contained in the interval $[-Mc^{n},Mc^{n}]$. Therefore, since
the dimension of $V^{n}$ is less than $n^{d-1}$, it follows that
$|\tilde{A}|-I$ is a trace class operator.
\end{proof}

\begin{thm}\label{thm:linear_map}
Let $V$ and $W$ be homogeneous varieties in $\mathbb{B}_{d}$ and $\mathbb{B}_{d'}$, respectively, with decompositions into (not necessarily irreducible)
subvarieties $V=V_{1}\cup\ldots\cup V_{k}$ and $W=W_{1}\cup\ldots\cup W_{k}$ with the property that $\mathrm{span}(W_{i})\cap\mathrm{span}(W_{j})=\{0\}$
whenever $i\ne j$. Suppose that there is a linear map $A:\mathbb{C}^{d}\to\mathbb{C}^{d'}$
such that $A(V_{i})=W_{i}$ and such that the
restriction of $A$ to $\mathrm{span}(V_{i})$ is isometric for all $i$,  $1\leq i\leq k$. Then for $p\geq1$, $\mathcal{F}_{W}$ is $p$-essentially
normal if and only if $\mathcal{F}_{V}$ is.
\end{thm}

\begin{proof}
Let $\tilde{A}:\mathcal{F}_{V}\to\mathcal{F}_{W}$ be as in Proposition
\ref{prop:quantitative}, so that we can write $\tilde{A}=U+T$, where
$U:\mathcal{F}_{V}\to\mathcal{F}_{W}$ is a unitary operator and $T:\mathcal{F}_{V}\to\mathcal{F}_{W}$
is a trace class operator. Then the identity $\tilde{A}M_{f}=M_{f\circ A^{*}}\tilde{A}$
implies that for every polynomial $f$ in $\mathcal{F}_{V}$, 
\[
M_{f\circ A^{*}}(U+T)=(U+T)M_{f},
\]
and hence that
\[
M_{f\circ A^{*}}=UM_{f}U^{*}+TM_{f}U^{*}-M_{f\circ A}TU^{*}.
\]
Therefore, for polynomials $f$ and $g$ in $\mathcal{F}_{V}$, we
can write
\[
[M_{f\circ A^{*}}^{*},M_{g\circ A^{*}}]=U[M_{f}^{*},M_{g}]U^{*}+R,
\]
where $R$ is a trace class operator. Letting $f$ and $g$ be coordinate functions, it immediately follows that $\mathcal{F}_V$ is $p$-essentially normal if $\mathcal{F}_W$ is. To obtain the converse, assume without loss of generality that $\mathbb{C}^{d'} = \mathrm{span}(W)$. Then $A$ is surjective, hence $A^*$ is left invertible. Let $B$ be a left inverse of $A^*$. Put $f = z_i \circ B$ and $g = z_j \circ B$, where $z_i$ and $z_j$ are considered as coordinate functions in $\mathbb{C}^{d'}$. Then $f$ and $g$ are linear combinations of coordinate function in $\mathbb{C}^d$. Now if $\mathcal{F}_V$ is $p$-essentially normal then $[M^*_f, M_g] \in \mathcal{L}^p$, whence $[M^*_{z_i}, M_{z_j}] = [M_{f\circ A^{*}}^{*},M_{g\circ A^{*}}]  \in \mathcal{L}^p$. Thus $\mathcal{F}_W$ is $p$-essentially normal. 
\end{proof}

\section{Decompositions of varieties and essential normality}\label{sec:decomposability}

\subsection{A refinement of a lemma} \label{subsec:refinement}
\begin{lem}
\label{lem:refinement}Let $I$ be a homogeneous ideal of $\mathbb{C}[z]$
and let $P$ denote the projection onto $\mathcal{F}_{I}$. Then
for $p>\dim I$, $\mathcal{F}_{I}$ is $p$-essentially normal if
and only if the commutator $[S_i,P]$ belongs to $\mathcal{L}^{2p}$
for each $1\leq i\leq d$.\end{lem}
\begin{rem}
A slightly weaker form of this conjecture, holding only for $p>d$ instead of $p>\dim I$, is well known (see, e.g., \cite[Proposition 4.2]{Arv07}).
\end{rem}
\begin{proof}
In \cite{Arv98}, it is shown that 
\begin{equation}
\|[S_{i}^{*},S_{j}]\mid_{\mathbb{H}_{n}}\|\leq2/(n+1).\label{eq:p-ineq-1}
\end{equation}
It follows that $\mathrm{trace}(|[S_{i}^{*},S_{j}]|^{p})<\infty$
for all $p>d$, since
\[
\mathrm{trace}(|[S_{i}^{*},S_{j}]|^{p})\leq\sum_{n=0}^{\infty}\frac{2^{p}\dim\mathbb{H}_{n}}{(n+1)^{p}},
\]
and this is finite for $p>d$, since $\dim\mathbb{H}_{n}=O(n^{d-1}).$

Write $T_i = S^I_i = P S_i P$, $i=1, \ldots, d$. Since $\mathcal{F}_{I}^{\perp}$ is an invariant subspace for the
$d$-shift,
\[
[T_{i}^{*},T_{j}]-P[S_{i}^{*},S_{j}]P=-P S_i^{*} (I-P) S_{j}P=-[P,S_{i}]^{*}[P,S_{j}],
\]
which we can rewrite as 
\begin{equation}
[T_{i}^{*},T_{j}]= P[S_{i}^{*},S_{j}]P - [P,S_{i}]^{*}[P,S_{j}].\label{eq:p-ineq-2}
\end{equation}
By (\ref{eq:p-ineq-1}) we know that $\|[S_{i}^{*},S_{j}]\mid_{\mathbb{H}_{n}}\|=O(n^{-1})$,
so it follows that there is a constant $M>0$ such that 
\[
\mathrm{trace}(|P[S_{i}^{*},S_{j}]P|^{p})\leq M\sum_{n=0}^{\infty}\frac{\dim(\mathbb{H}_{n}\ominus I_{n})}{n^{p}},
\]
and this is finite for $p>\dim I$. Therefore, $P[S_{i}^{*},S_{j}]P$
belongs to $\mathcal{L}^{p}$ for every $p>\dim I$. Furthermore,
for every $p\geq1$, $[P,S_{i}]$ belongs to $\mathcal{L}^{2p}$ for all $i$
if and only if $[P,S_{i}]^{*}[P,S_{i}]$ belongs to $\mathcal{L}^{p}$ for all $i,j$. The desired result now follows from (\ref{eq:p-ineq-2}).
\end{proof}

\subsection{Decomposability and essential normality}\label{subsec:decomposability}
\begin{lem}
\label{lem:sum-closed-comm}Let $M_{1},\ldots,M_{k}$ be subspaces
of a Hilbert space $\mathcal{H}$. For $p\geq1$, suppose that the
projections $P_{M_{1}},\ldots,P_{M_{k}}$ each commute modulo $\mathcal{L}^{p}$
with an operator $T$ in $\mathrm{B}(\mathcal{H})$. If the algebraic
sum $M_{1}+\ldots+M_{k}$ is closed, then the projection $P_{M_{1}+\ldots+M_{k}}$
onto the subspace $M_{1}+\ldots+M_{k}$ also commutes modulo $\mathcal{L}^{p}$
with $T$.\end{lem}
\begin{proof}
The proof of this result follows the outline of the proof of \cite[Theorem 3.3]{Ken12} or \cite[Theorem 4.4]{Arv07}.
\end{proof}

\begin{prop}
\label{prop:ideals}Let $I_{1},\ldots,I_{k}$ be homogeneous ideals of $\mathbb{C}[z_{1},\ldots,z_{d}]$.
\begin{enumerate}
\item If $\mathcal{F}_{I_{1}},\ldots,\mathcal{F}_{I_{k}}$ are $p$-essentially
normal for $p>\max\{\dim I_{1},\ldots,\dim I_{k}\}$, and the algebraic
sum $\mathcal{F}_{I_{1}}^{\perp}+\ldots+\mathcal{F}_{I_{k}}^{\perp}$
is closed, then $\mathcal{F}_{I_{1}+\ldots+I_{k}}$ is also $p$-essentially
normal.
\item If $\mathcal{F}_{I_{1}},\ldots,\mathcal{F}_{I_{k}}$ are $p$-essentially
normal for $p>\dim I_{1}\cap\ldots\cap I_{k}$, and the algebraic
sum $\mathcal{F}_{I_{1}}+\ldots+\mathcal{F}_{I_{k}}$ is closed, then
$\mathcal{F}_{I_{1}\cap\ldots\cap I_{k}}$ is also $p$-essentially
normal.
\end{enumerate}
\end{prop}
\begin{proof}
First, note that the submodule $\mathcal{F}_{I_{1}+\ldots+I_{k}}$
is the orthogonal complement of the algebraic sum $\mathcal{F}_{I_{1}}^{\perp}+\ldots+\mathcal{F}_{I_{k}}^{\perp}$,
and the submodule $\mathcal{F}_{I_{1}\cap\ldots\cap I_{k}}$ is the
closure of the algebraic sum $\mathcal{F}_{I_{1}}+\ldots+\mathcal{F}_{I_{k}}$.

If $\mathcal{F}_{I_{1}},\ldots,\mathcal{F}_{I_{k}}$ are $p$-essentially
normal for $p>\max\{\dim I_{1},\ldots,\dim I_{k}\}$, then by Lemma
\ref{lem:refinement}, each of the commutators $[S_{i},P_{\mathcal{F}_{I_{j}}}^{\perp}]$
belongs to $\mathcal{L}^{2p}$ for $1\leq i\leq d$ and $1\leq j\leq k$.
If the algebraic sum $\mathcal{F}_{I_{1}}^{\perp}+\ldots+\mathcal{F}_{I_{k}}^{\perp}$
is closed, then since 
\[
\mathcal{F}_{I_{1}}^{\perp}+\ldots+\mathcal{F}_{I_{k}}^{\perp}=\mathcal{F}_{I_{1}+\ldots+I_{k}}^{\perp},
\]
Lemma \ref{lem:sum-closed-comm} implies that the commutators $[S_{i},P_{\mathcal{F}_{I_{1}+\ldots+I_{k}}^{\perp}}]$
also belong to $\mathcal{L}^{2p}$, and hence that the commutators
$[S_{i},P_{\mathcal{F}_{I_{1}+\ldots+I_{k}}}]$ belong to
$\mathcal{L}^{2p}$ for $1\leq i\leq d$. Therefore, since $\dim(I_{1}+\ldots+I_{k}) \leq \max\{\dim I_{1},\ldots,\dim I_{k}\}$,
it follows from Lemma \ref{lem:refinement} that $\mathcal{F}_{I_{1}+\ldots+I_{k}}$
is also $p$-essentially normal.

If $\mathcal{F}_{I_{1}},\ldots,\mathcal{F}_{I_{k}}$ are $p$-essentially
normal for $p>\dim I_{1}\cap\ldots\cap I_{k}$, and the algebraic
sum $\mathcal{F}_{I_{1}}+\ldots+\mathcal{F}_{I_{k}}$ is closed, then
the proof that $\mathcal{F}_{I_{1}\cap\ldots\cap I_{k}}$ is also
$p$-essentially normal follows in the same way after noting that
\[
\mathcal{F}_{I_{1}}+\ldots+\mathcal{F}_{I_{k}}=\mathcal{F}_{I_{1}\cap\ldots\cap I_{k}},
\]
and that $\dim(I_{1}\cap\ldots\cap I_{k})\geq\max\{\dim I_{1},\ldots,\dim I_{k}\}$.
\end{proof}

\begin{prop}
\label{prop:varieties}Let $V_{1},\ldots,V_{k}$ be homogeneous varieties
in $\mathbb{B}_{d}$.
\begin{enumerate}
\item For $p>\max\{\dim V_{1},\ldots,\dim V_{k}\}$, if $\mathcal{F}_{V_{1}},\ldots,\mathcal{F}_{V_{k}}$
are $p$-essentially normal and the algebraic sum $\mathcal{F}_{V_{1}}^{\perp}+\ldots+\mathcal{F}_{V_{k}}^{\perp}$
is closed, then $\mathcal{F}_{V_{1} \cap \ldots \cap V_{k}}$is also $p$-essentially
normal.
\item For $p>\dim V_{1}\cup\ldots\cup V_{k}$, if $\mathcal{F}_{V_{1}},\ldots,\mathcal{F}_{V_{k}}$
are $p$-essentially normal and the algebraic sum $\mathcal{F}_{V_{1}}+\ldots+\mathcal{F}_{V_{k}}$
is closed, then $\mathcal{F}_{V_{1}\cup\ldots\cup V_{k}}$ is $p$-essentially
normal.
\end{enumerate}
\end{prop}
\begin{proof}
The proof of this result follows immediately from Proposition \ref{prop:ideals}
using the correspondence between ideals of $\mathbb{C}[z_{1},\ldots,z_{d}]$
and varieties in $\mathbb{C}^{d}$.
\end{proof}

\subsection{Some decomposable varieties} \label{subsec:examples}

The following theorem was proved by Michael Hartz in \cite{Har12}. We shall say that $V$ is a {\em linear subspace} in $\mathbb{B}_d$ if $V = L \cap \mathbb{B}_d$ where $L \subseteq \mathbb{C}^d$ is a subspace. 
\begin{thm}[Hartz]
\label{thm:sum-subspaces}
Let $V_{1},\ldots,V_{k}$ be linear subspaces in
$\mathbb{B}_{d}$. Then the algebraic sum $\mathcal{F}_{V_{1}}+\ldots+\mathcal{F}_{V_{k}}$
is closed.
\end{thm}

We can also handle the following additional case.

\begin{thm}
\label{thm:disjoint-varieties} 
Let $V_{1},\ldots,V_{n}$ be homogeneous varieties in $\mathbb{B}_{d}$.
Suppose that $\mathrm{span}(V_{i})\cap\mathrm{span}(V_{j})=\{0\}$
whenever $i\ne j$. Then the algebraic sum $\mathcal{F}_{V_{1}}+\ldots + \mathcal{F}_{V_{n}}$
is closed.
\end{thm}

\begin{proof}
We can suppose that each of the varieties $V_1,\ldots,V_n$ are nonempty. For $1\leq i\leq d$, let $L_{i}=\mathrm{span}(V_{i})$.
Then, as in the proof of Proposition \ref{prop:quantitative}, since $L_{1},\ldots,L_{n}$ are finite dimensional and disjoint, if we let 
\[
c=\max\{\cos(L_{i},L_{j})\mid i\ne j\},
\]
then $0\leq c<1$. Following the proof of \cite[Lemma 7.11]{DRS11}, this implies that
\[
\cos(V_{i}^{k},V_{j}^{k})\leq c^{k},
\]
which we can rewrite as
\begin{equation}
\sup\{|\langle x_{i},x_{j}\rangle|/(\|x_i\|\|x_j\|) \mid 0 \ne x_{i}\in V_{i}^{k},\ 0 \ne  x_{j}\in V_{j}^{k},\ i\ne j\}\leq c^{k}\label{eq:small-angle} .
\end{equation}
Let $V=V_{1}\cup\ldots\cup V_{n}$, and define an operator $T:\mathcal{F}_{V_{1}}\oplus\ldots\oplus\mathcal{F}_{V_{n}}\to\mathcal{F}_{V}$
by
\[
T(x_{1},\ldots,x_{n})=x_{1}+\ldots+x_{n},\quad(x_{1},\ldots,x_{n})\in\mathcal{F}_{V_{1}}\oplus\ldots\oplus\mathcal{F}_{V_{n}}.
\]
Then the range of $T$ is precisely $\mathcal{F}_{V_{1}}+\ldots+\mathcal{F}_{V_{n}}$,
and hence we will be done if we can prove that $T$ has closed range.

Note that $T$ is graded, in the sense that it maps $V_{1}^{k}\oplus\ldots\oplus V_{n}^{k}$
to $V^{k}$. For $(x_{1},\ldots,x_{n})$ in $V_{1}^{k}\oplus\ldots\oplus V_{n}^{k}$, the inequality (\ref{eq:small-angle}) implies that
\begin{eqnarray*}
\|T(x_{1},\ldots,x_{n})\|^{2} & = & \|x_{1}+\ldots+x_{n}\|^{2}\\
 & = & \sum_{i=1}^{n}\|x_{i}\|^{2}+\sum_{i\ne j}\langle x_{i},x_{j}\rangle\\
 & \geq & \sum_{i=1}^{n}\|x_{i}\|^{2}-\sum_{i\ne j}|\langle x_{i},x_{j}\rangle|\\
 & \geq & \sum_{i=1}^{n}\|x_{i}\|^{2}-c^{k}\sum_{i\ne j}\|x_{i}\|\|x_{j}\|\\
 & \geq & (1-c^{k}(n-1))\sum_{i=1}^{n}\|x_{i}\|^{2}\\
 & = & (1-c^{k}(n-1))\|(x_{1},\ldots,x_{n})\|^{2}.
\end{eqnarray*}
Therefore, for sufficiently large $k$, $T$ is uniformly bounded below on the
subspaces $V_{1}^{k}\oplus\ldots\oplus V_{n}^{k}$. Since each of these subspaces is finite dimensional, it follows that $T$ has closed range.
\end{proof}

\section{Applications} \label{sec:applications}

We now present two classes of examples for which our results imply the Arveson--Douglas conjecture.

\begin{thm}\label{thm:disjoint}
Let $V_{1},\ldots,V_{k}$ be homogeneous varieties in $\mathbb{B}_{d}$
such that $\mathrm{span}(V_{i})\cap\mathrm{span}(V_{j})=\{0\}$ whenever
$i\ne j$, and let $V = V_1 \cup \ldots \cup V_k$. Let $p > \dim V$, and suppose that $\mathcal{F}_{V_{1}},\ldots,\mathcal{F}_{V_{k}}$ are all $p$-essentially normal. Then $\mathcal{F}_{V}$ is also $p$-essentially normal.
\end{thm}

\begin{proof}
This result follows immediately from (2) of Proposition \ref{prop:varieties} and Theorem \ref{thm:disjoint-varieties}. However, we will present a different and more constructive proof as an application of Proposition \ref{prop:quantitative} and Theorem \ref{thm:linear_map}.

Let $L_j = \mathrm{span}(V_j)$ for $j=1, \ldots, k$ and define $d_j = \dim L_j$. Put $D = d_1 + \ldots + d_k$, and let $\{e_1, \ldots, e_D\}$  be some orthonormal basis in $\mathbb{C}^D$. Consider the subspaces of $\mathbb{C}^D$ given by $K_1 = \mathrm{span}\{e_1, \ldots, e_{d_1}\}$, $K_2 = \mathrm{span}\{e_{d_1 + 1}\ldots, e_{d_2}\}$, etc., up to $K_k$. Let $A: \mathbb{C}^D \rightarrow \mathbb{C}^d$ be a linear map that takes $K_j$ isometrically onto $L_j$ for all $j=1, \ldots, k$. Now define a homogeneous variety $W$ by  
\[
W = (A\big|_{K_1})^{-1}(V_1) \cup \ldots \cup (A\big|_{K_k})^{-1}(V_k) .
\]
For $j=1, \ldots, k$, the variety $W_j := (A\big|_{K_j})^{-1}(V_j)$ is unitarily equivalent to $V_j$, and therefore the Hilbert module $\mathcal{F}_{W_j}$ is unitarily equivalent to $\mathcal{F}_{V_j}$. It follows from the assumptions that $\mathcal{F}_{W_j}$ is $p$-essentially normal for all $j$. If we show that $\mathcal{F}_{W}$ is $p$-essentially normal, then Theorem \ref{thm:linear_map} will imply that so is $\mathcal{F}_{V}$. But in the situation where the components $W_j$ all lie in mutually orthogonal subspaces it is straightforward to check directly that $\mathcal{F}_{W}$ is essentially normal, so we are done. 
\end{proof}

Finally, let us observe that the Arveson--Douglas conjecture holds for any variety which is a union of subspaces. 

\begin{thm}\label{thm:subspaces}
Let $V_{1},\ldots,V_{k}$ be linear subspaces in $\mathbb{B}_{d}$. Then $\mathcal{F}_{V_{1}\cup\ldots\cup V_{k}}$ is $p$-essentially
normal for all $p>\dim V_1 \cup \ldots \cup V_k = \max \{\dim V_1, \ldots, \dim V_k\}$.  
\end{thm}

\begin{proof}
This follows from (2) of Proposition \ref{prop:varieties}, from Theorem \ref{thm:sum-subspaces}, and from the known result that, for a subspace $V$, $\mathcal{F}_V$ is $p$-essentially normal for $p>\dim V$ (this last fact is \cite[Proposition 5.3]{Arv98}, together with the observation theat $\mathcal{F}_V$ is unitarily equivalent to $H^2_{\dim V}$). 
\end{proof}

\begin{rem}
A very special case of Theorems \ref{thm:disjoint} and \ref{thm:subspaces}  is that every quotient module associated with a $1$-dimensional homogeneous variety is $p$-essentially normal
for all $p>1$. This special case is a known result, and was obtained
by different techniques in \cite[Proposition 4.1]{GW08}. \end{rem}

\subsection*{Acknowledgment} The authors would like to thank the anonymous referee for helpful comments.

\end{document}